\documentclass[11pt]{article} 

\usepackage{amssymb}
\usepackage{amsmath,amsthm,amsfonts,epsfig,setspace}

\usepackage[utf8]{inputenc}
\usepackage{amsmath}
\usepackage{amsfonts}
\usepackage{cases}

\usepackage{graphicx} 
\usepackage{subcaption}
\usepackage{color}

\numberwithin{equation}{section}

\usepackage{geometry} 
\usepackage{graphicx} 
\usepackage{subcaption}

\numberwithin{equation}{section}

\newcommand{\nm}{\noalign{\smallskip}}
\newcommand{\beq}{\begin{equation}}
\newcommand{\eeq}{\end{equation}}
\newcommand{\eqnref}[1]{(\ref {#1})}
\newcommand{\ds}{\displaystyle}
\newcommand{\p}{\partial}
\newcommand{\pd}[2]{\frac {\p #1}{\p #2}}

\newtheorem{theorem}{Theorem}[section]
\newtheorem{cor}[theorem]{Corollary}
\newtheorem{lemma}[theorem]{Lemma}
\newtheorem{prop}[theorem]{Proposition}

\newtheorem{remark}{Remark}
\newtheorem{asump}{Assumption}[section]

\newcommand{\R}{\mathbb{R}}
\newcommand{\f}{\frac}

\newcommand{\Bx}{{x}}
\newcommand{\By}{{y}}

\newcommand{\RR}{\mathbb{R}}

\newcommand{\Scal}{\mathcal{S}}
\newcommand{\Kcal}{\mathcal{K}}
\newcommand{\Acal}{\mathcal{A}}

\title{Double-negative acoustic metamaterials\thanks{\footnotesize The work of Hyundae Lee was supported by National Research Fund of Korea (NRF-2015R1D1A1A01059357, NRF-2017R1A4A1014735). The work of Hai Zhang was partially supported by Research Grant Council of Hong Kong (GRF grant 16304517)  and 
startup fund R9355 from HKUST.}}

\author{
Habib Ammari\thanks{\footnotesize Department of Mathematics, 
ETH Z\"urich, 
R\"amistrasse 101, CH-8092 Z\"urich, Switzerland (habib.ammari@math.ethz.ch, brian.fitzpatrick@sam.math.ethz.ch, sanghyeon.yu@sam.math.ethz.ch).} \and Brian Fitzpatrick\footnotemark[2] \and  Hyundae Lee\thanks{\footnotesize  Department of Mathematics, Inha University,  253 Yonghyun-dong Nam-gu,  Incheon 402-751,  Korea (hdlee@inha.ac.kr).}  \and Sanghyeon Yu\footnotemark[2] \and Hai Zhang\thanks{\footnotesize 
Department of Mathematics, 
 HKUST,  Clear Water Bay, Kowloon, Hong Kong (haizhang@ust.hk).}}

\date{} 
  
\begin{document}
\maketitle
\begin{abstract}
The aim of this paper is to provide a mathematical theory for understanding the mechanism behind the double-negative refractive index phenomenon in bubbly fluids. The design of double-negative metamaterials generally requires the use of two different kinds of subwavelength resonators, which may limit the applicability of double-negative metamaterials. Herein we rely on media that consists of only a single type of resonant element, and show how to turn the acoustic metamaterial with a single negative effective property obtained in [H. Ammari and H. Zhang, Effective medium theory for acoustic waves in bubbly fluids near Minnaert resonant frequency. SIAM J. Math. Anal., 49 (2017), 3252--3276.] into a negative refractive index metamaterial, which refracts waves negatively, hence acting as a superlens. Using bubble dimers made of two identical bubbles, it is proved that both the effective mass density and the bulk modulus of the bubbly fluid can be negative near the anti-resonance of the two hybridized Minnaert resonances for a single constituent bubble dimer. 
A rigorous justification of the Minnaert resonance hybridization, in the case of a  bubble dimer in a homogeneous medium, is established. The acoustic properties of a single bubble dimer are analyzed. Asymptotic formulas for the two hybridized Minnaert resonances are derived. Moreover, it is proved that the bubble dimer can be approximated by a point scatterer with monopole and dipole modes. For an appropriate volume fraction of bubble dimers with certain conditions on their configuration, a double-negative effective medium when the frequency is near the anti-resonance of the hybridized Minnaert resonances can be obtained.  

\end{abstract}

\def\keywords2{\vspace{.5em}{\textbf{  Mathematics Subject Classification
(MSC2000).}~\,\relax}}
\def\endkeywords2{\par}
\keywords2{35R30, 35C20.}

\def\keywords{\vspace{.5em}{\textbf{ Keywords.}~\,\relax}}
\def\endkeywords{\par}
\keywords{bubble, Minnaert resonance, hybridization, homogenization, double-negative metamaterial.}

\section{Introduction}

Metamaterials are man-made composite media structured on a scale much smaller than a wavelength. They raise the possibility of an unprecedented level of control when it comes to engineering the propagation of waves. A particularly interesting capability is the prospect of focusing or imaging beyond the diffraction limit. Metamaterials can manipulate and control sound waves in ways that are not possible in conventional materials. Their study 
 has also drawn increasing interest in recent times due to their potential application in cloaking \cite{fink, rev1,rev2}.  

Metamaterials can be assembled into structures (typically periodic but not necessarily so) that are similar to continuous materials, yet have unusual wave properties that differ substantially from those of conventional media.  Subwavelength resonators are the building blocks of metamaterials. Because of the subwavelength scale of the resonators, it is possible to describe the marcoscopic behavior of a metamaterial using homogenization theory, and this results in an effective medium having negative or high contrast parameters. Metamaterials with negative or high contrast refractive indices offer new possibilities for imaging and for the control of waves at deep subwavelength scales \cite{alu}.  

In acoustics, it is known that air bubbles are subwavelength resonators \cite{Minnaert1933}. Due to the high contrast between the air density inside and outside an air bubble in a fluid,  a quasi-static acoustic resonance known as the Minnaert resonance occurs \cite{H3a}. At or near this resonant frequency, the size of a bubble can be up to three orders of magnitude smaller than the wavelength of the incident wave, and the bubble behaves as a strong monopole scatterer of sound. The Minnaert resonance phenomenon makes air bubbles good candidates for acoustic subwavelength resonators.  They have the potential to serve as the basic building blocks for acoustic metamaterials, which include bubbly fluids \cite{leroy1, leroy2, leroy3}. This motivated our series of bubble studies \cite{H3a,AFLYZ,Ammari_David,Ammari_David2,Ammari_Hai}. We refer to \cite{H3a} for a rigorous mathematical treatment of Minnaert resonance and the derivation of the monopole approximation in the case of a single, arbitrary shaped bubble in a homogeneous medium. 

As shown in \cite{Ammari_Hai}, around the Minnaert resonant frequency, an effective medium theory can be derived in the dilute regime. Furthermore, {above} the Minnaert resonant frequency, the real part of the effective bulk {{modulus}} is negative, and consequently the bubbly fluid behaves as a diffusive medium for the acoustic waves. Meanwhile, below the Minnaert resonant frequency, with an appropriate bubble volume fraction, a high contrast effective medium can be obtained, making the sub-wavelength focusing or super-focusing of waves achievable \cite{Ammari2015_a}. These properties show that the bubbly fluid functions like an acoustic metamaterial, and indicate that a sub-wavelength bandgap opening occurs at the Minneaert resonant frequency \cite{leroy1}. We remark that such behavior is rather analogous to the coupling of electromagnetic waves with plasmonic nanoparticles, which results in effective negative or high contrast dielectric constants for frequencies near the plasmonic resonance frequencies \cite{pierre,matias1,matias2}.

In \cite{AFLYZ}, the opening of a sub-wavelength phononic bandgap is demonstrated by considering a periodic arrangement of bubbles and exploiting their Minnaert resonance.  
It is shown that there exists a subwavelength band gap in such a bubbly crystal. This subwavelength band gap is mainly due to the cell resonance of the bubbles in the quasi-static regime and is quite different from the usual band gaps in photonic/phononic crystals, where the gap opens at a wavelength which is comparable to the period of the structure \cite{figotin, arma, akl}. In \cite{homogenization}, the homogenization theory of the bubbly crystal near the frequency where the band gap opens is further investigated. It is shown that the band gap opens at the  corner (edge in two dimensions) of the Brillouin zone. Moreover, explicit formulas for the Bloch eigenfunctions are derived. This makes both the homogenization theory and the justification of the superfocusing phenomenon in the non-dilute case possible. We also refer to \cite{Ammari_David} for the related work on bubbly metasurfaces in which a homogenization theory is developed for a thin layer of periodically arranged bubbles mounted on a perfectly reflecting surface.

In this paper, we aim to understand the mechanism behind the double-negative refractive index phenomenon in bubbly fluids. The design of double-negative metamaterials generally requires the use of two different kinds of building blocks or specific subwavelength resonators presenting multiple overlapping resonances. Such a requirement limits the applicability of double-negative metamaterials. Herein we rely on media that consists of only a single type of resonant element, and show how to turn the acoustic metamaterial with a single negative effective property obtained in \cite{Ammari_Hai} into a negative refractive index metamaterial, which refracts waves negatively, hence acting as a superlens \cite{pendry2000, milton1994, milton2007}.

Our main result is to prove that, using bubble dimers, the effective mass density and bulk modulus of the bubbly fluid can both be negative over a non empty range of frequencies. A bubble dimer is a system consists of two identical separated bubbles.   It features two slightly different subwavelength resonances, called the hybridized Minnaert resonances.  
We establish  a rigorous mathematical justification of the Minnaert resonance hybridization in the case of a  bubble dimer in a homogeneous medium. We analyze the acoustic properties of the bubble dimer, derive asymptotic formulas for the two resonances, and prove that the bubble dimer can be approximated as a point scatterer with monopole and dipole modes. The hybridized Minnaert resonances are fundamentally different modes. The first mode is, as in the case of a single bubble,  a monopole mode, while the second mode is a dipole mode. The resonance associated with the dipole mode is usually referred to as the anti-resonance. For an appropriate volume fraction, when the excitation frequency is close to the anti-resonance, we obtain a double-negative effective mass density and bulk modulus for bubbly media consisting of a large number of bubble dimers with certain conditions on their distribution. The dipole modes in the background medium contribute to the effective mass density while the monopole modes contribute to the effective bulk modulus.

The paper is organized as follows. In Section \ref{sec:setup} we introduce some preliminaries on layer potentials.  In Section \ref{sec:pre} we describe the hybridization phenomenon for a bubble dimer, prove that two subwavelength resonances occur, and compute their asymptotic expansions in terms of the mass density contrast. In Section \ref{sec:approx} we prove that the bubble dimer can be approximated as the sum of a monopole source and a dipole source. In Section \ref{sec:homo} we derive a double-negative effective medium theory for bubbly media. Finally, in Section \ref{sec:num} we compute effective mass density and bulk modulus dispersion curves near the hybridized Minnaert resonances to illustrate the double-negative property of bubbly media.

\section{Preliminaries} \label{sec:setup}

\subsection{Layer potentials}

For a given bounded domain $D$ in  $\mathbb{R}^{3}$, with Lipschitz boundary $\partial D$, the single layer potential of the density function $\varphi\in L^{2}(\partial D)$ is defined by
\begin{equation*}
\mathcal{S}_{D}^k[\varphi](\Bx):=\int_{\partial D} G( \Bx-\By, k ) \varphi(\By)d\sigma(\By),~~~\Bx\in \mathbb{R}^{3},
\end{equation*}
where $G(\Bx, k)$ is the fundamental solution to $\triangle + k^2 $, i.e., 
\begin{equation}
G(\Bx, k) = 
-\frac{1}{4\pi|\Bx|} {\exp (ik|\Bx|)}.
\end{equation}

The following jump relation holds:
\begin{equation} \label{jump1}
\left. \frac{\partial}{\partial\nu} \mathcal{S}_{D}^k[\varphi]\right|_{\pm}(\Bx)=\left (\pm\frac{1}{2}I+\mathcal{K}_{D}^{k,*}\right)[\varphi](\Bx),~~~\Bx\in \partial D,
\end{equation}
where the Neumann-Poincar\'e operator $\mathcal{K}_{D}^{k,*}$ is defined by 
\begin{equation*}
\mathcal{K}_{D}^{k,*}[\varphi](x)=\int_{\partial D}\pd {G(\Bx -\By, k)}{\nu(\Bx)} \varphi(\By)d\sigma(\By),~~~\Bx\in \partial D.
\end{equation*}
Here $\pd{}{\nu}$ denotes the normal derivative on $\partial D$, and the subscripts $+$ and $-$ indicate the limits from outside and inside $D$, respectively.
 
We make use of low frequency asymptotic expansions of the layer potentials. For a small parameter $\epsilon$, $G(\Bx,\epsilon k)$ can be expanded as
 \beq
 G(\Bx, \epsilon k) = -\frac{1}{4\pi} \sum_{n=0}^\infty  \frac{(i\epsilon k)^n}{n!} | \Bx |^{n-1},
 \eeq
and it follows that
\begin{align}\label{layer_asymp}
 \Scal_{D}^{\epsilon k} [\phi]= \sum_{n=0}^\infty (\epsilon k )^n \Scal_{D}^n [\phi], \quad
 \Kcal_{D}^{\epsilon k}[\phi] = \sum_{n=0}^\infty (\epsilon k)^n \Kcal_{D}^n [\phi],
 \end{align}
where
 \begin{align}
 \Scal_{D}^n [\phi] (\Bx)& := -\frac{i^n}{4\pi n!}  \int_{\p D} |\Bx - \By|^{n-1} \phi(\By) d\sigma(\By),\\
 \Kcal_{D}^n [\phi] (\Bx)& := -\frac{i^n (n-1)}{4\pi n!}  \int_{\p D} {\langle\, \Bx -\By, \nu_\Bx \rangle }{|\Bx - \By|^{n-3}}\phi(\By) d\sigma(\By).
\end{align}
It is known that $\mathcal{S}_D^0 : L^2(\p D)\rightarrow H^1(\p D)$ is invertible and that its inverse is bounded.

Throughout the paper we assume that $D$ has two connected components $D_1$ and $D_2$, and we use the following notation:
\begin{equation} \label{capaciatnce}
\int_{\p D_1 - \p D_2} f(\By)d\sigma(\By) = \int_{\p D_1} f(\By)d\sigma(\By)- 
\int_{\p D_2} f(\By)d\sigma(\By), \end{equation}
where $f$ is a function defined on $\p D = \p D_1\cup \p D_2$. 
 
Finally, we present some useful formulas which are frequently used in the sequel. 

 \begin{lemma}\label{lem:trick_layer}
The following identities hold for any $\phi\in L^2(\p D)$: for j = 1, 2,
\begin{itemize}
\item[(i)]
$\int_{\p D_j}   \left(\frac{1}{2}I-\mathcal{K}_{ D}^{0,* } \right) \phi(y) d\sigma(y) =0;$

\item[(ii)]
$\int_{\p D_j}   \left(\frac{1}{2}I+\mathcal{K}_{ D}^{0,* } \right) \phi(y) d\sigma(y) =\int_{\p D_j}\phi(y) d\sigma(y);$
\item[(iii)]
$
 \int_{\p D_j} \Kcal_D^2 [\phi] = -\int_{D_j}\Scal_D^0 [\phi];
$
\item[(iv)]
$
 \int_{\p D_j} \Kcal_D^3[\phi] = {i|D_j|}/{(4\pi)} \int_{\p D} \phi.
$

\end{itemize}
 \end{lemma}

 \begin{proof}
 (i) follows from the jump relation $(-\frac{1}{2}I+\mathcal{K}_{ D}^{0,* })[\phi] = \p \mathcal{S}_D^0 [\phi]/\p\nu|^-_{\p D}$ and the fact that $\mathcal{S}_D^0 [\phi]$ is harmonic in $D$. (ii) immediately follows from (i). For (iii), we have 
  \begin{align*}
 \int_{\p D_j} \Kcal_D^2 \phi &= \frac{1}{8\pi} \int_{\p D_j} \pd{}{\nu_\Bx} \int_{\p D} |\Bx- \By| \phi (\By)
 =\frac{1}{8\pi} \int_{ D_j} \int_{\p D} \frac{2}{|\Bx- \By|} \phi (\By)= -\int_{D_j}\Scal_D^0 [\phi]. 
 \end{align*}
Finally, it holds that
 \begin{align*}
 \int_{\p D_j} \Kcal_D^3[\phi] &=\frac{i}{12\pi} \int_{\p D_j} \int_{\p D} \langle\, \Bx -\By, \nu_\Bx \rangle \phi(\By) = \frac{i|D_j|}{4\pi} \int_{\p D} \phi, 
 \end{align*}
which proves (iv).
\end{proof}

 \subsection{Capacitance coefficients}\label{subsec:cap}
Let $\psi_1$, $\psi_2 \in L^2(\p D)$ be given by
\begin{align}
\mathcal{S}_D^0[\psi_1] = \begin{cases} 1 \quad \mbox{on}~\partial D_1,\\ 0 \quad \mbox{on}~\partial D_2, \end{cases} \quad\mathcal{S}_D^0[\psi_2] = \begin{cases} 0\quad \mbox{on}~\partial D_1,\\ 1\quad \mbox{on}~\partial D_2 . \end{cases}
\end{align}
Then, using \eqnref{jump1}, it is easy to check that 
 \beq\label{ker_1_2_K} \ker \left( -\frac{1}{2} I + \mathcal{K}_D^{0,*}\right)= \mbox{span } \{\psi_1, \psi_2 \}.
 \eeq
We define the capacitance coefficients matrix $C=(C_{ij})$ by
$$ C_{ij} := - \int_{\partial D_j} \psi_i,\quad i,j=1,2.$$
We remark that the matrix $C$ is positive definite and symmetric. 

Suppose $D_1$ and $D_2$ are identical balls of radius $r_0$ separated by a distance $d_0$ such that the distance between the centers of the balls is $d_0+2 r_0$. Then 
$C_{11}=C_{22}, \ C_{12}=C_{21}, \ C_{11}>0$, and $C_{12}<0$.
Explicit formulas for the capacitance coefficients for two balls can be obtained using bispherical coordinates \cite{Lek}. We have the following result:
\begin{align*}
C_{11} &= 8\pi \alpha \sum_{n=0}^\infty \frac{e^{(2n+1)T}}{e^{2(2n+1)T}-1},
\\
C_{12} &= -8\pi \alpha\sum_{n=0}^\infty \frac{1}{e^{2(2n+1)T}-1},
\\
C_{21}&=C_{12}, \quad C_{22}=C_{11},
\end{align*}
where
$$
\alpha=\sqrt{d_0(r_0+d_0/4)}
,\quad
T=\sinh^{-1}(r_0/a).
$$


\begin{lemma} \label{lem-trick}
Suppose that $D_1$ and $D_2$ are two identical balls.
Then the following identities hold for any $\phi \in L^2(\p D)$:
\begin{enumerate}

\item[(i)]
$\int_{\p D_j} \Kcal_D^3[\psi_1-\psi_2]d\sigma=0,\quad j=1,2;$

\item[(ii)]
$
\int_{\p D_1 - \p D_2} (\mathcal{S}_D^0)^{-1} \mathcal{S}_D^1[\phi] d\sigma =0;
$

\item[(iii)]
$\int_{\partial D_1- \partial D_2} (\mathcal{S}_D^0)^{-1}[y_i]d\sigma(y) = \int_{\partial D} y_i(\psi_1-\psi_2)d\sigma(y).$


\end{enumerate}
\end{lemma} 

\begin{proof}
(i) and (ii) follow from Lemma \ref{lem:trick_layer} (iv), the definition of $\mathcal{S}_D^1$ and the symmetry of $D_1\cup D_2$. 
For (iii), by letting $\phi_i(x):=(\mathcal{S}_D^0)^{-1}[y_i](x)$ and using the jump relations \eqnref{jump1} and the fact that $\mathcal{S}_D^0[\phi]$ is harmonic in $D$, we can check that
\begin{align*}
\int_{\partial D_1- \partial D_2} (\mathcal{S}_D^0)^{-1}[y_i]d\sigma(y) &=  \int_{\partial D_1- \partial D_2} 1 \cdot \left(\frac{\p \mathcal{S}_D^0[\phi_i]}{\p\nu}\Big|_+ - \frac{\p \mathcal{S}_D^0[\phi_i]}{\p\nu}\Big|_-  \right)
\\
&=\int_{\partial D_1- \partial D_2} 1 \cdot \frac{\p \mathcal{S}_D^0[\phi_i]}{\p\nu}\Big|_+ 
=\int_{\partial D} \mathcal{S}_D[\psi_1-\psi_2] \cdot \frac{\p \mathcal{S}_D^0[\phi_i]}{\p\nu}\Big|_+ .
\end{align*}
Then, since  and $\mathcal{S}_D^0[\phi_i]|_{\p D}=y_i$, the Green's identity yields (iii).
\end{proof}

\section{Resonance for a dimer consisting of two identical bubbles } \label{sec:pre}
In this section, we consider the quasi-static resonances of a bubble dimer. 
Throughout, we denote by $D$ the normalized bubble dimer which consists of two identical balls $D_1$ and $D_2$, both with volume one. We assume that the two balls $D_1$ and $D_2$ are symmetric with respect to the origin. Moreover,  $D_1$ and $D_2$ are aligned with respect to the $x_1$-axis. For a general dimer featuring two identical balls, we call the unit direction vector along which the two balls are aligned the orientation of the dimer.
Let $(\kappa_b, \rho_b)$ and $(\kappa, \rho)$ be the bulk modulus and density of air and water, respectively.
Let $u^{in}$ be the incident wave which we assume to be a plane wave for simplicity:
$$
u^{in}(\Bx)= e^{ i \omega \sqrt{\rho/\kappa} \Bx \cdot \theta},
$$
where $\theta$ is a unit vector in $\RR^3$. 
Then the acoustic wave propagation can be modeled as 
\beq \label{HP1}
\begin{cases}
\nabla \cdot \left( \frac{1}{\rho_b}  \chi_D + \frac{1}{\rho} \chi(\RR^3\setminus \overline{D})\right) \nabla u + \omega^2 \left(  \frac{1}{\kappa_b}  \chi_D + \frac{1}{\kappa} \chi(\RR^3\setminus \overline{D})\right)u=0 \quad \mbox{in} ~ \RR^3,\\
u -u^{in} \mbox{ satisfies the Sommerfeld radiation condition at infinity.}
\end{cases}
\eeq
Recall that the  Sommerfeld radiation condition at infinity can be expressed by
$$
\big| \frac{\partial }{\partial |x|} (u - u^{in}) (\Bx) - \omega \sqrt{\rho/\kappa} (u - u^{in}) (\Bx)\big| =O(\frac{1}{|\Bx|^2}),
$$
uniformly in $\Bx/|\Bx|$ as $|\Bx| \rightarrow + \infty$.

Let 
\begin{equation} \label{v}  v=\sqrt{\frac{\kappa}{\rho}}, \ ~ v_b=\sqrt{\frac{\kappa_b}{\rho_b}}, \ ~k=\omega \sqrt{\frac{\rho}{\kappa}}, \ ~ k_b=\omega \sqrt{\frac{\rho_b}{\kappa_b}}.\end{equation}
We assume that \begin{equation} \label{delta} 
\delta:=\frac{\rho_b}{\rho} \ll 1 \end{equation}  and \begin{equation} 
v, v_b=O(1).\end{equation}

From, for example, \cite{AK04},
we know that the solution to \eqnref{HP1} can be represented using
the single layer potentials $\Scal_D^{k_b}$ and $\Scal_D^{k}$ as 
follows:
 \beq\label{helmScal}
 u(\Bx) = \begin{cases}
 u^{in}(\Bx)+\Scal_D^{ k}[\psi](\Bx),\quad \Bx\in\RR^3\setminus\bar{D},\\
 \Scal_D^{k_b} [\phi](\Bx),\quad \Bx\in D,
 \end{cases}
 \eeq
where the pair $(\phi,\psi)\in L^2(\p D)\times L^2(\p D)$ is
a solution to \beq\label{phipsi} \ \left\{
\begin{array}{l}
\ds \Scal^{k_b}_D[\phi] - \Scal_D^k[\psi] = u^{in}\\
\nm
\ds \frac{1}{\rho_b}\left (-\frac{1}{2}I+\mathcal{K}_{ D}^{k_b,*}\right)[\phi]
-\frac{1}{\rho}\left (\frac{1}{2}I+\mathcal{K}_{ D}^{k,*}\right)[\psi]=\frac{1}{\rho}\pd{u^{in}}{\nu}
\end{array}\right.
\quad\mbox{on }\p D. \eeq  
We denote by
 \beq
 \Acal_\delta^\omega := 
 \begin{bmatrix}
 \Scal^{k_b}_D & -\Scal^{ k}_D\\
\left(-\frac{1}{2}I+\mathcal{K}_{D}^{k_b,* }  \right)& -\delta \left (\frac{1}{2}I+\mathcal{K}_{D}^{k,*}\right) 
 \end{bmatrix}.
 \eeq
Then (\ref{phipsi}) can be written as
\[
\mathcal{A}_\delta^\omega \begin{pmatrix} \phi \\ \psi \end{pmatrix}=
\begin{pmatrix} u^{in} \\ \delta\pd{u^{in}}{\nu} \end{pmatrix}.
\]
It is well-known that the above integral equation has a unique solution 
for all real frequencies $\omega$. 

The resonance of the bubble dimer $D$ in the scattering problem \eqnref{HP1} can be defined as all the complex numbers $\omega$ with negative imaginary part such that there exists a nontrivial solution to the following equation:
\beq\label{eq_A_delta_omega_zero}
\mathcal{A}_\delta^\omega \begin{pmatrix} \phi \\ \psi \end{pmatrix}=0.
\eeq
These can be viewed as the characteristic values of the operator-valued analytic function $\mathcal{A}_\delta^\omega$ (with respect to $\omega$); see \cite{akl}.

It can be shown that $\omega=0$  is a characteristic value of $\Acal_0^\omega$ when $\delta=0$. Then Gohberg-Sigal theory \cite{akl} tells us that there exists a characteristic value $\omega_0=\omega_0(\delta)$ such that $\omega_0(0)=0$ and $\omega_0$ depends on $\delta$ continuously. We call this characteristic value the quasi-static resonance (or Minnaert resonance) \cite{H3a}. We now present the main result on the quasi-static resonances of the bubble dimer $D$.

\begin{theorem}
\label{sphere_res} 
There are two quasi-static resonances with positive real part for the bubble dimer $D$. Moreover, they have the following asymptotic expansions as $\delta$, defined by (\ref{delta}), goes to zero: 
\begin{align}
 \omega_1 &=  \sqrt{ \delta v_b^2 (C_{11}+ C_{12})}-i\tau_1 \delta +O(\delta^{3/2}), \label{eq-omega1}\\
\omega_2&= \sqrt{\delta v_b^2 (C_{11}-C_{12})}  + \delta^{3/2} \hat\eta_1 + i\delta^2 \hat\eta_2 + O(\delta^{5/2}) \label{eq-omega2},
\end{align}
where  
$$ 
\tau_1 = \frac{v_b^2}{4\pi  v}(C_{11}+C_{12})^2,
$$ 
and $\hat\eta_1$ and $\hat\eta_2$ are real numbers which are determined by $D$, $v$, and $v_b$. The two resonances $\omega_1$ and $\omega_2$ are called the hybridized resonances of the bubble dimmer $D$. The resonance $\omega_2$ is referred to as the anti-resonance.

\end{theorem} 
\begin{proof} Step 1. Suppose that $(\phi,~\psi)$ is a nontrivial solution to \eqnref{eq_A_delta_omega_zero} for some small $\omega=\omega(\delta)$.

Using the asymptotic expansions \eqnref{layer_asymp} of the single layer potential and the Neumann-Poincar\'e operator, together with the fact that $k=O(\omega)$, we have 
  \begin{equation}\label{asymp_int_eq}\begin{cases}
  &\Scal_D^0[\phi - \psi]+  k_b \Scal_D^1[\phi] -  k \Scal_D^1[\psi] =O(\omega^2),\\
&  \left(-\frac{1}{2}I+\mathcal{K}_{ D}^{0,* } + k_b^2 \mathcal{K}_{D}^{2 } +k_b^3\mathcal{K}_D^3\right)[\phi]-\delta\left (\frac{1}{2}I+\mathcal{K}_{ D}^{0,*}\right) [\psi] =O(\omega^4+\delta\omega^2).
  \end{cases}\end{equation}
 
 From the first equation of \eqnref{asymp_int_eq} and the definition of $\mathcal{S}_D^1$, it holds that
 $$
 \psi  =  \phi-\frac{1}{4\pi i}(\mathcal{S}_D^0)^{-1}
 \left( k_b \int_{\p D} \phi - k \int_{\p D}\psi \right) + O(\omega^2).
 $$
 Then, from the fact that $\phi-\psi = O(\omega)$ and the definition of $\psi_j$, we obtain
 \begin{equation} \label{plugint} 
 \psi=\phi+\frac{(k_b -k)}{4\pi i} (\psi_1+\psi_2)\int_{\partial D} \phi + O(\omega^2).
 \end{equation}
 Plugging  (\ref{plugint}) into the second equation of \eqnref{asymp_int_eq}, we get
 
  \begin{align}\label{approx_eq}
 &\left(-\frac{1}{2}I+\mathcal{K}_{ D}^{0,* } \right)[\phi] + \left(k_b^2 \mathcal{K}_{D}^{2 } + k_b^3 \mathcal{K}_{D}^{3 }-\delta\left (\frac{1}{2}I+\mathcal{K}_{ D}^{0,*}\right) \right)[\phi]
 \nonumber\\
& \qquad\qquad\qquad\qquad\qquad -\frac{\delta(k_b -k)}{4\pi i} (\psi_1+\psi_2)\int_{\partial D} \phi =O(\omega^4+\delta \omega^2),
  \end{align}
  where we have used the identity
  \begin{align*}
  \psi_j &= \frac{\p\mathcal{S}_D}{\p\nu}\Big|_{\p D}^+[\psi_j]- \frac{\p\mathcal{S}_D}{\p\nu}\Big|_{\p D}^-[\psi_j]
  \\&=\frac{\p\mathcal{S}_D}{\p\nu}\Big|_{\p D}^+[\psi_j]-0=\left(\frac{1}{2}I+\mathcal{K}_{ D}^{0,*}\right) [\psi_j].  
  \end{align*}
In view of \eqnref{ker_1_2_K}, a nontrivial solution $\phi$ to \eqnref{approx_eq} can be written as
\begin{align}\label{phi_decomp}
\phi = a\psi_1 + b \psi_2 + O(\omega^2+\delta),\end{align} 
for some nontrivial constants $a,~b$ with $|a|+|b|=O(1)$.

Step 2. 
Recall that $|D_1|= |D_2|=1$, $C_{11}=C_{22}$, and $C_{12}=C_{21}<0$. By integrating \eqnref{approx_eq} over $\partial D_j,~j=1,2$, and then using Lemma \ref{lem:trick_layer}, we have, up to an error of order  $O(\omega^4+\delta \omega^2)$, 
\begin{equation}\label{sys_eq_ball}\begin{cases}
\ds -\frac{ \omega^2}{v_b^2}  a- \frac{i\omega^3 }{4\pi v_b^3} C(a,b)+\delta(a C_{11} + b C_{12} )+\frac{i(v_b^{-1}- v^{-1})\delta\omega}{4\pi }  (C_{11}+C_{12}) C(a,b)=0, \\
\ds - \frac{ \omega^2}{v_b^2}  b  - \frac{i\omega^3}{4\pi  v_b^3}  C(a,b)+\delta(a C_{12} + b C_{11} )+\frac{i( v_b^{-1}- v^{-1})\delta\omega}{4\pi } (C_{12}+C_{11}) C(a,b)=0,
\end{cases}\end{equation}
where 
$$C(a,b):=a(C_{11}+C_{12})+b(C_{21}+C_{22}).$$
Then it follows that
for $\delta \ll 1$, the characteristic values of 
$\Acal^\omega_{\delta}$ are given by
 \beq \label{frequency_form}
 \omega_* = \omega_0+ O(\delta^{3/2}),
 \eeq
for some $\omega_0=O(\sqrt{\delta})$. Solving \eqnref{sys_eq_ball} for $\omega$, we find two characteristic values with positive real parts:
$$ \omega_1 = \sqrt{\delta v_b^2 (C_{11}+C_{12})}+O(\delta), \quad \omega_2 = \sqrt{\delta v_b^2 (C_{11}-C_{12})}+O(\delta),$$
where the corresponding $(a,b)$'s are given by
 \begin{equation}\label{ab_asymp}
 (a,b)=(1,1),~(1,-1),
 \end{equation}
up to an error of order $O(\delta)$. Plugging these values for $a$, $b$ and $\omega_j$ into \eqnref{sys_eq_ball} and solving for the $O(\delta)$ term in $\omega_j$, we obtain 
$$
\omega_1 = \sqrt{\delta v_b^2 (C_{11}+C_{12})}-i \tau_1 \delta +O(\delta^{3/2}), \quad \omega_2= \sqrt{\delta v_b^2 (C_{11}-C_{12})} + 0\cdot \delta + O(\delta^{3/2}). 
$$
The proof of the estimate (\ref{eq-omega1}) for $\omega_1$ is concluded. We refine the estimate for $\omega_2$ in the next step.


Step 3.  Note that $\Im \omega_2$ is of order $\delta^{3/2}$. We perform a further calculation to find an explicit formula for $\Im  \omega_2$. 
Since $(a,b)=(1,-1)+O(\delta)$, we write
$$ \phi=\psi_1-\psi_2 + \delta \phi_1 + \delta^{3/2} \phi_2 + \delta^2 \phi_3+ \delta^{5/2}\phi_4+O(\delta^3),$$
 and $$\omega_2= \delta^{1/2}\eta_0+\delta^{3/2} \eta_1 +\delta^2 \eta_2+ O(\delta^{5/2}),$$
where $\eta_0:= \sqrt{ v_b^2 (C_{11}-C_{12})}$. From the symmetry of $D$, we can normalize $\phi$ so that
$$
\phi_j \perp \mbox{span }\{\psi_1, \psi_2\}=\ker \left( -\frac{1}{2} I + \mathcal{K}_B^{0,*}\right), \quad \mbox{for } j=1, 2, 3, 4, \quad \mbox{in } L^2(\p D). 
$$

Step 4. 
By using \eqnref{layer_asymp}, we have 
  \begin{equation}\label{asymp_int_eq3}\begin{cases}
&\ds \Scal_D^0[\phi - \psi]+  k_b \Scal_D^1[\phi] -  k \Scal_D^1[\psi] + k_b^2 \Scal_D^2[\phi] -  k^2 \Scal_D^2[\psi]+ k_b^3 \Scal_D^3[\phi] -  k^3 \Scal_D^3[\psi]=O(\delta^2),\\
& \ds \left(-\frac{1}{2}I+\mathcal{K}_{ D}^{0,* } + k_b^2 \mathcal{K}_{D}^{2 } +k_b^3\mathcal{K}_D^3+k_b^4\mathcal{K}_D^4+k_b^5\mathcal{K}_D^5\right)[\phi]-\delta\left (\frac{1}{2}I+\mathcal{K}_{ D}^{0,*}+ k^2\mathcal{K}_D^2+k^3\mathcal{K}_D^3\right) [\psi] =O(\delta^3).
  \end{cases}\end{equation}
Hence, we get
\begin{align*}
\ds \psi&= (\Scal_D^0+  k \Scal_D^1 +  k^2 \Scal_D^2 +k^3 \Scal_D^3)^{-1} (\Scal_D^0+  k_b \Scal_D^1 + k_b^2 \Scal_D^2+k_b^3 \Scal_D^3)[\phi]+O(\delta^2)\\
\ds& =\phi+(\Scal_D^0+  k \Scal_D^1 +k^2\Scal_D^2)^{-1} (  (k_b-k) \Scal_D^1 + (k_b^2 -k^2) \Scal_D^2+ (k_b^3 -k^3) \Scal_D^3)[\phi]+O(\delta^2)\\
\ds &=\phi+(\Scal_D^0+  k \Scal_D^1)^{-1}( (k_b^2 -k^2) \Scal_D^2+ (k_b^3 -k^3) \Scal_D^3)[\psi_1-\psi_2]+\delta(k_b -k) (\Scal_D^0)^{-1}\Scal_D^1[\phi_1]+O(\delta^2)\\
\ds &=\phi+(k_b^2 -k^2) (\Scal_D^0)^{-1} \Scal_D^2[\psi_1-\psi_2] + (k_b^3 -k^3) (\Scal_D^0)^{-1} \Scal_D^3[\psi_1-\psi_2] \\
\ds &\quad -k(k_b^2- k^2)(\Scal_D^0)^{-1}\Scal_D^1(\Scal_D^0)^{-1}\Scal_D^2[\psi_1-\psi_2]  +\delta(k_b -k) (\Scal_D^0)^{-1}\Scal_D^1[\phi_1]+O(\delta^2).
\end{align*}
Plugging this into the second equation of \eqnref{asymp_int_eq3} and equating terms of the same order of $\delta$, $\delta^{3/2}$, $\delta^2$, $\delta^{5/2}$, we arrive at
\begin{itemize}
\item $O(\delta)$-terms:
\begin{align}
 \left(-\frac{1}{2}I+\mathcal{K}_{ D}^{0,* }\right) [\phi_1]+ \eta_0^2 v_b^{-2}\Kcal_D^2[\psi_1-\psi_2]  -(\psi_1-\psi_2)=0, \label{phi1_eq}
  \end{align}
whence $\phi_1$ is uniquely determined. 
 
 \item  $O(\delta^{3/2})$-terms:
\begin{align}
 \left(-\frac{1}{2}I+\mathcal{K}_{ D}^{0,* }\right) [\phi_2]+ \eta_0^3 v_b^{-3} \Kcal_D^3[\psi_1-\psi_2]=0, \label{phi2_eq}
 \end{align}
whence $\phi_2$ is uniquely determined. 
 
\item $O(\delta^2)$-terms:
 \begin{align*}
 &  \left(-\frac{1}{2}I+\mathcal{K}_{ D}^{0,* } \right)[\phi_3]+ 2\eta_0\eta_1 v_b^{-2} \Kcal_D^2[ \psi_1-\psi_2] + \eta_0^2 v_b^{-2} \Kcal_D^2[ \phi_1] +\eta_0^4 v_b^{-4} \Kcal_D^4[\psi_1-\psi_2] \\
& -\left (\frac{1}{2}I+\mathcal{K}_{ D}^{0,*}\right)[ \eta_0^2(v_b^{-2} - v^{-2}) (\Scal_D^0)^{-1}\Scal_D^2[\psi_1-\psi_2] +\phi_1] - v^{-2} \eta_0^2 \Kcal_D^2 [\psi_1-\psi_2]=0
\end{align*}

\item $O(\delta^{5/2})$-terms:
\begin{align*}
&   \left(-\frac{1}{2}I+\mathcal{K}_{ D}^{0,* } \right)[\phi_4]+2\eta_0\eta_2 v_b^{-2} \Kcal_D^2[\psi_1-\psi_2]+\eta_0^2v_b^{-2} \Kcal_D^2 [\phi_2]+\eta_0^3 v_b^{-3} \Kcal_D^3 [\phi_1]\\
&+ 3 \eta_0^2\eta_1  v_b^{-3} \Kcal_D^3[\psi_1-\psi_2]+\eta_0^5v_b^{-5} \Kcal_D^5 [\psi_1-\psi_2]\\
&- \eta_0^3 \left(\frac{1}{2}I+\mathcal{K}_{ D}^{0,* } \right) \bigg[ (v_b^{-3} -v^{-3}) (\Scal_D^0)^{-1} \Scal_D^3[\psi_1-\psi_2] \\
& -v^{-1}(v_b^{-2} - v^{-2})(\Scal_D^0)^{-1}\Scal_D^1(\Scal_D^0)^{-1}\Scal_D^2[\psi_1-\psi_2] \bigg]\\
 & - \eta_0^3 \left(\frac{1}{2}I+\mathcal{K}_{ D}^{0,* } \right) \left [ (v_b^{-1} -v^{-1}) (\Scal_D^0)^{-1}\Scal_D^1[\phi_1]  \right] - \left(\frac{1}{2}I+\mathcal{K}_{ D}^{0,* } \right)[\phi_2]-\eta_0^3 v^{-3} \Kcal_D^3[\psi_1-\psi_2]=0.
 \end{align*}
\end{itemize}

Step 5. We consider the terms of order $O(\delta^2)$. 
Integrating over $\p D_1 - \p D_2$ and using Lemma \ref{lem-trick}, we get
\begin{align}
-2\eta_0\eta_1 v_b^{-2}  +  \int_{\p D_1 -\p D_2} \left(\eta_0^2 v_b^{-2}\Kcal_D^2[\phi_1] - \phi_1\right)+\eta_0^4 v_b^{-4} \int_{\p D_1-\p D_2}\Kcal_D^4[\psi_1-\psi_2] \nonumber\\
-\eta_0^2 (v_b^{-2} - v^{-2})\int_{\p D_1 -\p D_2} (\Scal_D^0)^{-1}\Scal_D^2[\psi_1-\psi_2] +  \eta_0^2 v^{-2}=0 .\label{eta1_eq}
\end{align}
It follows that $\eta_1$ is uniquely determined from the above equation. Moreover, we can check that $\eta_1$ is a real number. 
%
%
%

Step 6. Finally, we consider the $\delta^{5/2}$ terms. 
Integrating  over $\p D_1 -\p D_2$ and using Lemma \ref{lem-trick} again, we get
\begin{align}
&-4\eta_0\eta_2 v_b^{-2} +\int_{\p D_1 -\p D_2} \eta_0^2v_b^{-2} \Kcal_D^2 [\phi_2]-\phi_2+\eta_0^3 v_b^{-3} \Kcal_D^3 [\phi_1] \nonumber\\
&+ \int_{\p D_1 - \p D_2} \eta_0^5v_b^{-5} \Kcal_D^5 [\psi_1-\psi_2]-\eta_0^3(v_b^{-3} -v^{-3}) (\Scal_D^0)^{-1} \Scal_D^3[\psi_1-\psi_2] =0.\label{eta2_eq}
\end{align}
Therefore, $\eta_2$ is uniquely determined. We also note that $\eta_2$ is a purely imaginary number. Indeed, from \eqnref{phi1_eq}, \eqnref{phi2_eq}, we see that $\Im \phi_1=0, \Re \phi_2=0$. This combined with the integral representation of $\Kcal_D^2, \Kcal_D^3, \Kcal_D^5$ yields the desired result. 

%

Thus we conclude that the resonance frequency $\omega_2$ has the following asymptotic expansion:  
\begin{equation}
\omega_2= \sqrt{\delta v_b^2 (C_{11}-C_{12})}  + \delta^{3/2} \hat\eta_1 + i\delta^2 \hat\eta_2 + O(\delta^{5/2}),
\end{equation}
where $\hat\eta_1:=\eta_1$ and $\hat\eta_2:= -i \eta_2$ are real numbers given by \eqnref{eta1_eq} and \eqnref{eta2_eq}, respectively.\end{proof}

\begin{remark}
The above approach is applicable to a general bubble dimer which may consist of two non-identical spherical bubbles. 
\end{remark}

\section{The point scatterer approximation of  bubbles} \label{sec:approx} 
In this section we prove an approximate formula for the solution $u$ to the scattering problem for the bubble dimer $D=D_1\cup D_2$.  

We need the following lemma which can be proved using a simple symmetry argument.
\begin{lemma}
We have
$$ 
\int_{\partial D} y_2(\psi_1-\psi_2)d\sigma(y)=\int_{\partial D} y_3(\psi_1-\psi_2)d\sigma(y)=0,
$$ 
while 
$$
\int_{\partial D} y_1(\psi_1-\psi_2)d\sigma(y)= P, 
$$
for some nonzero real number $P$. 
\end{lemma} 
 
 In the next theorem, we prove that the bubble dimer can be approximated as a point scatterer with monopole and dipole modes. 
\begin{theorem} \label{thm:point_and_dipole_approx}
For $\omega=O(\delta^{1/2})$ and a given plane wave $u^{in}$, the solution $u$ to \eqnref{HP1} can be approximated as $\delta \rightarrow 0$ by 
\begin{equation}
u(\Bx)-u^{in}(\Bx)=  g^0(\omega)u^{in}(0) G(\Bx, k) +\nabla u^{in}(0)\cdot g^1(\omega) \nabla G(\Bx,  k) +O(\delta|x|^{-1}),
\end{equation}
when $|x|$ is sufficiently large, where
\begin{align}
&g^0(\omega):=\frac{C(1,1)}{1- \omega_1^2/\omega^2}(1+O(\delta^{1/2})),\\
&g^1(\omega)=(g^1_{ij}(\omega)),~g^1_{ij}(\omega):= \int_{\partial D}  (\mathcal{S}_D^0)^{-1}[x_i](y) y_j + \frac{\delta v_b^2}{|D|(\omega_2^2-\omega^2)} P^2\delta_{i,1}\delta_{j,1}.
\end{align}
\end{theorem}

\begin{proof} Step 1. 
 Let $(\phi,~\psi)$ be the solution to \eqnref{phipsi}. Using the asymptotic expansions \eqnref{layer_asymp}, we have 
  \begin{equation}\label{asymp_int_eq_inhom}\begin{cases}
& \ds \Scal_D^0[\phi - \psi]+  k_b \Scal_D^1[\phi] -  k \Scal_D^1[\psi] = u^{in}+O(\delta),\\
& \ds  \left(-\frac{1}{2}I+\mathcal{K}_{ D}^{0,* } + k_b^2 \mathcal{K}_{D}^{2 } +k_b^3\mathcal{K}_D^3\right)[\phi]-\delta\left (\frac{1}{2}I+\mathcal{K}_{ D}^{0,*}\right) [\psi] = \delta \pd{u^{in}}{\nu}+O(\delta^2),
  \end{cases}\end{equation}
where the remainders $O(\delta)$ and $O(\delta^2)$ are in the operator norm.

From the first equation of \eqnref{asymp_int_eq_inhom}, and following the same approach used to derive \eqnref{plugint}, we obtain
 \begin{equation}\label{psi_eq}
 \psi=\phi+\frac{(k_b -k)}{4\pi i} (\psi_1+\psi_2)\int_{\partial D} \phi - (\mathcal{S}_D^0)^{-1}[u^{in}]+O(\delta).
 \end{equation}
Plugging \eqnref{psi_eq} into the second equation of \eqnref{asymp_int_eq_inhom}, we get
 \begin{equation}\label{approx_eq_inhom2}
\mathcal{C}_\delta^\omega[\phi]=  -\delta\left (\frac{1}{2}I+\mathcal{K}_{ D}^{0,*}\right)  (\mathcal{S}_D^0)^{-1}[u^{in}]+ \delta \pd{u^{in}}{\nu} +O(\delta^2),
 \end{equation} 
where $\mathcal{C}_\delta^\omega$ is defined by
 \begin{align*} 
 \mathcal{C}_\delta^\omega[\phi]&:=\left(-\frac{1}{2}I+\mathcal{K}_{ D}^{0,* }[\phi]\right) + \left(k_b^2 \mathcal{K}_{D}^{2 } + k_b^3 \mathcal{K}_{D}^{3 }-\delta\left (\frac{1}{2}I+\mathcal{K}_{ D}^{0,*}\right) \right)[\phi]
 \\
 &\qquad -\frac{\delta(k_b -k)}{4\pi i} (\psi_1+\psi_2)\int_{\partial D} \phi.
 \end{align*}
 Note that $\mathcal{C}_\delta^\omega[\phi]$ is equal to the left-hand side of 
\eqnref{approx_eq}.

Step 2.  
Using the Taylor expansion of $u^{in}$ at the origin, and the fact that $\nabla u^{in} = O(\omega) = O(\delta^{1/2})$ and $\nabla^2 u^{in} = O(\omega^2) = O(\delta)$, the right-hand side of \eqnref{approx_eq_inhom2} can be approximated by
 \begin{align*}
& -\delta\left (\frac{1}{2}I+\mathcal{K}_{ D}^{0,*}\right)  (\mathcal{S}_D^0)^{-1}[u^{in}](\Bx)+ \delta \pd{u^{in}}{\nu}(\Bx)\\
 &= -\delta\left (\frac{1}{2}I+\mathcal{K}_{ D}^{0,*}\right)  (\mathcal{S}_D^0)^{-1}[u^{in}(0)+ \nabla u^{in}(0)\cdot \By](\Bx)+ \delta\nabla (u^{in}(0)) \cdot \nu(\Bx)+O(\delta^{3/2})\\
 &=  -\delta u^{in}(0) (\psi_1 +\psi_2) - \delta  (\mathcal{S}_D^0)^{-1}[\nabla u^{in}(0)\cdot \By] + 0 + O(\delta^{3/2}).
 \end{align*}
Therefore, \eqnref{approx_eq_inhom2} becomes
 \begin{equation}\label{asym_main_eq}
 \mathcal{C}_\delta^\omega[\phi]=-\delta u^{in}(0) (\psi_1 +\psi_2) - \delta  (\mathcal{S}_D^0)^{-1}[\nabla u^{in}(0)\cdot \By] + O(\delta^{3/2}).
 \end{equation}

Step 3. 
In the statement and proof of Theorem \ref{sphere_res}, we have verified that the characteristic values of $\mathcal{C}_\delta^\omega$ are given by \eqnref{eq-omega1} and \eqnref{eq-omega2}, and that their corresponding singular functions are $\phi_1$, $\phi_2$, i.e.,
 $$ \mathcal{C}_\delta^{\omega_1}[\phi_1]= \mathcal{C}_\delta^{\omega_2}[\phi_2]=0.$$
Recall from \eqnref{phi_decomp} and \eqnref{ab_asymp} that
$$ \phi_1= \psi_1+\psi_2 + O(\delta), ~ \phi_2=\psi_1-\psi_2+ O(\delta).$$ 

We decompose the solution $\phi\in L^2(\p D)$ to \eqnref{asym_main_eq} as
\begin{equation}\label{a_phi1_b_phi2_phi3}
\phi= a\phi_1 + b\phi_2 + \phi_3
\end{equation}
with $\langle \phi_1,\,\phi_3 \rangle=0$ and $\langle \phi_2,\,\phi_3 \rangle=0$, where $\langle \;,\,\; \rangle$ denotes the $L^2$-inner product on $\p D$.

We have
\begin{align}\label{asym_main_eq2}
&
a(\mathcal{C}_\delta^\omega-\mathcal{C}_\delta^{\omega_1})[\phi_1]+b (\mathcal{C}_\delta^\omega-\mathcal{C}_\delta^{\omega_2})[\phi_2] + \mathcal{C}_\delta^\omega[\phi_3]
\nonumber
\\
&
\qquad \qquad
=  -\delta u^{in}(0) (\psi_1 +\psi_2) - \delta  (\mathcal{S}_D^0)^{-1}[\nabla u^{in}(0)\cdot \By].
\end{align}

Since 
\begin{align*}
(\mathcal{C}_\delta^\omega-\mathcal{C}_\delta^{\omega_j})[\phi_j]=\frac{\omega^2-\omega_j^2}{v_b^2} \mathcal{K}_D^2[\phi_j]+ O(\delta^{3/2}),\quad j=1,2,\end{align*}
we have
$\|(\mathcal{C}_\delta^\omega-\mathcal{C}_\delta^{\omega_j})[\phi_j] \|=O(\delta)$, and hence we conclude from  \eqnref{asym_main_eq2} that
\begin{equation}\label{est1} \| \phi_3 \|= O((|a|+|b|+1)\delta).\end{equation}

Integrating \eqnref{asym_main_eq2} over $\partial D$ and then using Lemma \ref{lem:trick_layer}, we obtain
\begin{equation}\label{int1}
-a\left( \frac{\omega^2-\omega_1^2}{v_b^2} |D|+ O(\delta^{3/2})\right) + bO(\delta^2)+ \| \phi_3\| O(\delta) = 2\delta u^{in}(0) (C_{11}+C_{12}).
\end{equation}
Similarly, by integrating \eqnref{asym_main_eq2} over $\partial D_1 - \partial D_2$, we have 
\begin{equation}\label{int2}
a O(\delta^2)-b\left( \frac{\omega^2-\omega_2^2}{v_b^2} |D|+ O(\delta^{3/2})\right) + \| \phi_3\| O(\delta) = -\delta \int_{\partial D_1 - \partial D_2} (\mathcal{S}_D^0)^{-1}[\nabla u^{in}(0)\cdot \By].
\end{equation}
We observe that
\begin{equation}\label{a_b_phi3_estim}
a=O(1), \ ~b=O(\delta^{1/2}), \ ~\| \phi_3 \|= O(\delta).
\end{equation}

By solving \eqnref{int1} and \eqnref{int2} for $a$ and $b$, and then using \eqnref{eq-omega1}, we get
\begin{align}
&a=- \frac{2\delta v_b^2 (C_{11}+C_{12})}{|D|(\omega^2-\omega_1^2)} (u^{in}(0)+O(\delta))= - \frac{\omega_1^2+O(\delta^{3/2})}{\omega^2-\omega_1^2}  (u^{in}(0)+O(\delta)),\label{eq-a-asymp}\\
&b= \frac{\delta v_b^2}{|D|(\omega^2 - \omega_2^2)} \left (  \int_{\partial D_1 - \partial D_2} (\mathcal{S}_D^0)^{-1}[\nabla u^{in}(0)\cdot \By]+O(\delta)\right).\label{eq-b-asymp}
\end{align}

Now we can calculate $\psi$. By using \eqnref{psi_eq}, \eqnref{a_b_phi3_estim}, \eqnref{eq-a-asymp} and \eqnref{eq-b-asymp}, we obtain
\begin{align*}
\psi&= a(1+O(\delta^{1/2}))\phi_1+ b\phi_2+ \phi_3 - (\mathcal{S}_D^0)^{-1}[u^{in}(0)+ \nabla u^{in}(0)\cdot \By]+O(\delta)\nonumber\\
&=- \frac{\omega_1^2+O(\delta^{3/2})}{\omega^2-\omega_1^2}  (u^{in}(0)+O(\delta))(1+O(\delta^{1/2}))\phi_1\nonumber - u^{in}(0) \phi_1
\\
&\quad + b\phi_2  + (\mathcal{S}_D^0)^{-1}[\nabla u^{in}(0)\cdot \By]+O(\delta)
\nonumber
\\
&= \frac{\omega^2+O(\delta^{3/2})}{\omega_1^2-\omega^2}  (u^{in}(0)+O(\delta)) (1+O(\delta^{1/2})) \phi_1 +b\phi_2  -\nabla u^{in}(0)\cdot  (\mathcal{S}_D^0)^{-1}[\By] +O(\delta).
\end{align*}

Step 4.  
Finally, we consider the scattered field $u^s(\Bx):= u(\Bx)-u^{in}(x)=\Scal_D^{ k}[\psi](\Bx)$. It is enough to consider $\Scal_D^{ k}[\phi_1]$ and $\Scal_D^{ k}[\phi_2]$.

Note that $$G(\Bx- \By, k) = G(\Bx, k)-\nabla G(\Bx, k)\cdot \By + O(\delta |x|^{-1}).$$
Note also that, due to the symmetry of $D_1$ and $D_2$, we have
$$
\int_{\partial D}  y (\psi_1+\psi_2)(\By) d\sigma(\By) = 0, 
\quad \int_{\p D} \psi_1 - \psi_2 = 0.
$$
Therefore, for sufficiently large $|x|$, we have
\begin{align*}
\Scal_D^{ k}[\phi_1](\Bx) &= \Scal_D^{ k}[\psi_1+\psi_2](\Bx)+O(\delta |\Bx|^{-1})\\
&=\int_{\partial D} G(\Bx-\By, k) (\psi_1+\psi_2)(\By) d\sigma(\By)+O(\delta |\Bx|^{-1})\\
&=\int_{\partial D} G(\Bx, k) (\psi_1+\psi_2)(\By) d\sigma(\By) 
\\
&\quad -\nabla G(\Bx, k)\cdot \int_{\partial D}  y (\psi_1+\psi_2)(\By) d\sigma(\By)
  +O(\delta |\Bx|^{-1}),
\\
&=-2(C_{11}+C_{12}) G(\Bx, k) + 0 + O(\delta|\Bx|^{-1} ).
\end{align*}
Similarly, we have 
\begin{align*}
\Scal_D^{ k}[\phi_2](\Bx) &= \Scal_D^{ k}[\psi_1-\psi_2](\Bx)+O(\delta|x|^{-1})\\
&= \int_{\partial D} G(\Bx-\By, k) (\psi_1-\psi_2)(\By) d\sigma(\By)+O(\delta |x|^{-1})\\
&=-\left(\int_{\partial D} (\psi_1-\psi_2)\By\right)\cdot\nabla G(\Bx, k)  + O(\delta|x|^{-1}).
\end{align*}
 The proof is then complete. \end{proof}
 
\begin{cor}
For the rescaled bubble dimer $sR_dD$ with size $s$ and orientation $d$, where $R_d$ represents a rotation transform which aligns the bubble dimer $D$ in the direction $d$, we have
\begin{align}
\omega_1(\delta, sR_dD) &= \frac{1}{s} \omega_1(\delta, D), \\
\omega_2(\delta, sR_dD) &= \frac{1}{s} \omega_2(\delta, D),\\
g^0(\omega, \delta, sR_dD) &= \frac{2 (C_{11}+C_{12})s}{1- \omega_1(\delta,s R_d D)^2/\omega^2}(1+O(\delta^{1/2})) \label{sc_func_monopole} \\
g^1(\omega, \delta, sR_dD)&= s^3 \int_{\partial D}  (\mathcal{S}_D^0)^{-1}[(R_d x)_i](y) (R_d y)_j + \frac{\delta v_b^2s}{|D|(\omega_2(\delta,s R_d D)^2-\omega^2)} P^2d_id_j. \label{sc_func_dipole}
\end{align}
\end{cor}

\begin{proof} By a scaling argument, one can show that 
$$
C_{ij}(sD) = sC_{ij}(D), \quad P(sD) = s^2 P(D), $$
and $$  \int_{\partial (sR_dD)}  (\mathcal{S}_{sR_dD}^0)^{-1}[x_i](y) y_j = s^3 \int_{\partial D}  (\mathcal{S}_D^0)^{-1}[(R_d x)_i](y) (R_d y)_j .
$$
The proof is then complete. \end{proof}

\begin{remark}
The coefficients $C_{11}, C_{12}$ and $P$ can be explicitly computed using bispherical coordinates. Explicit formulas for $C_{11}$ and $C_{12}$ are given in subsection \ref{subsec:cap}.
Using the same approach as in the derivation of $C_{ij}$ \cite{Lek}, we can also obtain an explicit formula for $P$. It holds that
$$
P=-4\pi r_0 (r_0+d_0/2)-8\pi \alpha^2 \sum_{n=0}^\infty (2n+1)e^{-(2n+1)T} \coth((n+1/2)T).
$$

\end{remark}

\section{Homogenization theory} \label{sec:homo} 

We consider the scattering of an incident acoustic plane wave $u^{in}$ by $N$ identical bubble dimers with different orientations distributed in a homogeneous fluid in $\R^3$.  The $N$ identical bubble dimers are generated by scaling the normalized bubble dimer $D$ by a factor $s$, and then rotating the orientation and translating the center. More precisely,  the bubble dimers occupy the domain
$$
D^N:=\cup_{1\leq j \leq N}D_j^N,
$$ 
where $D_j^N=z_j^N + s R_{d_j^N}D$ for $1\leq j \leq N$, with $z_j^N$ being the center of the dimer $D_j^N$, $s$ being the characteristic size, and $R_{d_j^N}$ being the rotation in $\R^3$ which aligns the dimer $D_j^N$ in the direction $d_j^N$.  Here, $d_j^N$ is a vector of unit length in $\R^3$.

We assume that $0< s \ll 1$, $N\gg 1$ and that $\{z_j^N: 1\leq j \leq N\} \subset \Omega$ where $\Omega$ is a bounded domain.  
Let $u^{in}$ be the incident wave which we assume to be a plane wave for simplicity. The scattering of acoustic waves by the bubble dimers can be modeled by the following system of equations: 
\beq \label{eq-scattering}
\left\{
\begin{array} {ll}
&\nabla \cdot \f{1}{\rho} \nabla  u^N+ \frac{\omega^2}{\kappa} u^N  = 0 \quad \mbox{in } \R^3 \backslash D^N, \\
&\nabla \cdot \f{1}{\rho_b} \nabla  u^N+ \frac{\omega^2}{\kappa_b} u^N  = 0 \quad \mbox{in } D^N, \\
& u^N_{+} -u^N_{-}  =0   \quad \mbox{on } \partial D^N, \\
&  \f{1}{\rho} \f{\p u^N}{\p \nu} \bigg|_{+} - \f{1}{\rho_b} \f{\p u^N}{\p \nu} \bigg|_{-} =0 \quad \mbox{on } \partial D^N,\\
&  u^N- u^{in}  \,\,\,  \mbox{satisfies the Sommerfeld radiation condition}, 
  \end{array}
 \right.
\eeq
where $u^N$ is the total field and $\omega$ is the frequency.
Then the solution $u^N$ can be written as 
\beq \label{Helm-solution}
u^N(x) = \left\{
\begin{array}{lr}
u^{in} + \mathcal{S}_{D^N}^{k} [\psi^N], & \quad x \in \R^3 \backslash \overline{D^N},\\
\nm
\mathcal{S}_{D}^{k_b} [\psi_b^N] ,  & \quad x \in {D^N},
\end{array}\right.
\eeq
for some surface potentials $\psi, \psi_b \in  L^2(\p D^N)$. Here, we have used the notations
\begin{align*}
L^2(\p D^N) &= L^2(\p D^N_1) \times L^2(\p D^N_2) \times \cdots \times L^2(\p D^N_N), \\ 
\mathcal{S}_{D^N}^{k} [\psi^N]&= \sum_{1\leq j \leq N} \mathcal{S}_{D^N_j}^{k} [\psi_j^N], \\
\mathcal{S}_{D}^{k_b} [\psi_b^N]&=  \sum_{1\leq j \leq N} \mathcal{S}_{D^N_j}^{k} [\psi_{bj}^N].
\end{align*}

Using the jump relations for the single layer potentials, it is easy to derive that $\psi$ and $\psi_b$ satisfy the following system of boundary integral equations:
\beq  \label{eq-boundary}
\mathcal{A}^N(\omega, \delta)[\Psi^N] =F^N,  
\eeq
where
\[
\mathcal{A}^N(\omega, \delta) = 
 \begin{pmatrix}
  \mathcal{S}_{D^N}^{k_b} &  -\mathcal{S}_{D^N}^{k}  \\
  -\f{1}{2}Id+ \mathcal{K}_{D^N}^{k_b, *}& -\delta( \f{1}{2}Id+ \mathcal{K}_{D^N}^{k, *})
\end{pmatrix}, 
\,\, \Psi^N= 
\begin{pmatrix}
\psi_b^N\\
\psi^N
\end{pmatrix}, 
\,\,F^N= 
\begin{pmatrix}
u^{in}\\
\delta \f{\partial u^{in}}{\partial \nu}
\end{pmatrix}|_{\p D^N}.
\]

One can show that the scattering problem (\ref{eq-scattering}) is equivalent to the system of boundary integral equations (\ref{eq-boundary}) \cite{akl, book2}. Furthermore, it is well-known that there exists a unique solution to the scattering problem (\ref{eq-scattering}), or equivalently to the system (\ref{eq-boundary}). 
%


We are concerned with the case when there is a large number of small identical bubble dimers distributed in a bounded domain and the frequency of the incident wave is close to the hybridized Minnaert resonances for a single bubble dimer. 

We recall that for a bubble dimer $z+ sR_dD$, there exist two quasi-static resonances which are given by
\begin{align*}
\omega_1(\delta, z+ sR_dD) &= \frac{1}{s} \omega_1(\delta, D), \\
\omega_2(\delta, z+ sR_dD) &= \frac{1}{s} \omega_2(\delta, D).
\end{align*}
We are interested in the limit when the size $s$ tends to zero while the frequency is of order one. In order to fix the order of the resonant frequency, we make the following assumption.

\begin{asump} \label{asump-1}
$\delta = \mu^2 s^2$ for some positive number $\mu >0$. 
\end{asump}

As a result, the two hybridized resonances have the following asymptotic expansions:
\begin{align*}
\omega_1(\delta, D^N_j) &= \omega_{M,1}- i \tau_1 \mu^2 s +O(s^2), \\
\omega_2(\delta, D^N_j) &= \omega_{M,2}+ \mu^3 \hat \eta_1 s^2 - i \mu^4 \hat \eta_2 s^3 + O(s^4),
\end{align*}
where 
$$
\omega_{M,1} = v_b \mu \sqrt{ (C_{11}+ C_{12}) }, \quad
\omega_{M,2} = v_b \mu \sqrt{(C_{11}- C_{12}) }.
$$ 
Moreover, the monopole and dipole coefficients are given by
\begin{align}
&g^0(\omega, \delta, D^N_k):=\frac{2s(C_{11}+C_{12})}{1- \omega_1(\delta, D^N_k)^2/\omega^2}(1+O(\delta^{1/2})),\\
&g^1(\omega, \delta, D^N_k)=(g^1_{ij}(\omega, \delta, D^N_k)),
\end{align}
where
$$
g^1_{ij}(\omega, \delta, D^N_k):= s^3\int_{\partial D}  (\mathcal{S}_D^0)^{-1}[(R_{d_k^N} x)_i](y) (R_{d_k^N} y)_j  + \frac{\mu^2  v_b^2 s^3}{2(\omega_2(\delta, D^N_k)^2-\omega^2)} P^2(d_k^N)_i(d_k^N)_j.
$$

\begin{asump}\label{asump-2}
$\omega = \omega_{M,2} + as^2$ for some real number $a \neq \mu^3 \hat \eta_1 $. 
\end{asump}

Then 
\begin{align}
&g^0(\omega, \delta, D^N_k):=\frac{2s(C_{11}+C_{12})}{1- \omega_{M,1}^2/\omega_{M, 2}^2}(1+O(s)),\\
&g^1_{ij}(\omega, \delta, D^N_k):= \frac{\mu^2  v_b^2 s}{2|D|\omega_{M,2} \left((\mu^3 \hat \eta_1 -a) - i \mu^4 \hat \eta_2 s \right)} P^2(d_k^N)_i(d_k^N)_j + O(s^3).
\end{align}

We introduce the two constants  
$$
\tilde g^0 = \frac{2(C_{11}+C_{12})}{1- \omega_{M,1}^2/\omega_{M, 2}^2}, \quad \tilde g^1 = 
\frac{\mu^2  v_b^2 }{2|D|\omega_{M,2} (\mu^3 \hat \eta_1 -a)} P^2.
$$

We now impose conditions on the distribution of the bubble dimers. 
\begin{asump} \label{asump-3}
$sN = \Lambda $ for some positive number $\Lambda >0$. 
\end{asump}
Note that the volume fraction of the bubble dimers is of the order of $s^3N$. The above assumption implies that the bubble dimers are very dilute with the average distance between neighboring dimers being of the order of $\frac{1}{N^{1/3}}$.

\begin{asump} \label{asump-4} The bubble dimers are regularly distributed in the sense that
 \[
\min_{i \neq j } |z^N_i -z^N_j|  \geq  \eta N^{-\f{1}{3}},
\]
for some constant $\eta$ independent of $N$.  Here, $\eta N^{-\f{1}{3}}$ can be viewed as the minimum separation distance between neighbouring bubble dimers.
\end{asump}

In addition, we also make the following assumptions on the regularity of the distribution $\{z_j^N: 1\leq j \leq N\}$ and the orientation $\{d_j^N: 1\leq j \leq N\}$.

\begin{asump} \label{asump5-1}
There exists a function $\tilde V \in C^1(\bar{\Omega})$ such that for any $f \in C^{0, \alpha}(\Omega)$ with $0 < \alpha \leq 1$,
\beq 
\max_{1\leq j \leq N} | \f{1}{N} \sum_{i \neq j} G(z_j^N - z_i^N,k) f(z_i^N) -
\int_{\Omega} G(z_j^N - y,k) \tilde{V}(y) f(y) dy| \leq C \f{1}{N^{\f{\alpha}{3}}}\|f\|_{C^{0, \alpha}(\Omega)}  
\eeq
for some constant $C$ independent of $N$. 
\end{asump}

\begin{asump} \label{asump5-2}
There exists a matrix valued function $\tilde B \in C^1(\bar{\Omega})$ such that
for $f \in (C^{0, \alpha}(\Omega))^3$ with $0 < \alpha \leq 1$, 
\beq 
\max_{1\leq j \leq N} |\f{1}{N} \sum_{i\neq j} (f(z_i^N )\cdot d_i^N)( d_i^N \cdot \nabla G(z_i^N - z_j^N,k)) -
\int_{\Omega} f(y) \tilde B \nabla G(y-z_j^N,k)  dy| \leq C \f{1}{N^{\f{\alpha}{3}}}\|f\|_{C^{0, \alpha}(\Omega)}  
\eeq
for some constant $C$ independent of $N$. 
\end{asump}

%

\begin{remark}
If we let $\{z_j^N: 1\leq j \leq N\}$ be uniformly distributed, then $\tilde V$ is a positive constant in $\Omega$. We can also let the orientation be uniformly distributed in the unit sphere in the sense that the average of the matrix $d^N_j(d^N_j)^T$ in any neighborhood of any point in $\Omega$ tends to a multiple of the identity matrix as $N$ tends to infinity. In that case, $\tilde B$ is a positive constant multiple of the identity matrix at each point. 
\end{remark}

\subsection{The homogenized equations}

In the same spirit as the point interaction approximation \cite{hamdache, Ammari_Hai,figari,ozawa, papanicolaou}, we now formally derive the homogenized equation. 

For $1\leq j \leq N$, denote by  
\begin{align} \label{eq-system1}
u^{i,N}_j&= u^{in} + \sum_{i \neq j} \mathcal{S}_{D_i^N}^k [\psi_i^N], \\
u^{s,N}_j &= \mathcal{S}_{D_j^N}^k [\psi_j^N].
\end{align}
It is clear that $u^{i,N}_j$ is the total incident field which impinges on the bubble $D_j^N$, and $u^{s,N}_j$ is the corresponding scattered field.  
Denote by
$$
\Omega_N = \Omega \backslash \cup_{1\leq j \leq N} B(z_j^N, \sqrt{s}).
$$
Note that the volume fraction of the set $\cup_{1\leq j \leq N} B(z_j^N, \sqrt{s})$ is of the order 
$$
O(N\cdot s^{\frac{3}{2}})= O(N\cdot s) \cdot s^{\frac{1}{2}}, 
$$
which tends to zero as $N \to \infty$ under Assumption \ref{asump-3}. 
We have
$$
u^N(x) = u^{in} + \sum_{1\leq k \leq N} u^{s, N}_k(x)=u^{i, N}_j(x) + u^{s, N}_j(x), \mbox{  for each } 1\leq j \leq N \,\,\,\, \mbox{and}\,\, x\in \Omega_{N}. 
$$

\begin{prop}  \label{lem-pointapprox1}
Under Assumptions \ref{asump-1}, \ref{asump-2}, \ref{asump-3}, we have that, for $x\in \Omega_N$,
$$
u^{s,N}_j(x) \approx  g^0(\omega, \delta, D^N_j)u^{i, N}_j(z_j^N) G (\Bx-z_j^N,k) +\nabla u^{i, N}_j(z_j^N)\cdot g^1(\omega, \delta, D^N_j) \nabla G (\Bx-z_j^N,k). 
$$
\end{prop}

We further assume that
\begin{asump} \label{asum}
There exists some macroscopic field $u \in C^{1, \alpha}(\Omega)$ such that
$$
u^N (x) \to u(x) \,\,\, \mbox{for}\,\, x\in \Omega_{N}.
$$
Here the convergence is understood in the sense that for any $\epsilon >0$, there exists $N_0$ such that for all $N\geq N_0$, we have
$$
\|u^N (x) - u(x)\|_{C^{1, \alpha}(\Omega_{N})} \leq \epsilon.
$$
\end{asump}

The above assumption implies in particular that 
$
u^{i, N}_j(z_j^N) \to u(z_j^N).
$
Therefore

\begin{align*}
& g^0(\omega, \delta, D^N_j)u^{i, N}_j(z_j^N) G (\Bx-z_j^N,k) \to 
\frac{1}{N} \Lambda u(z_j^N) \tilde g^0 G(x- z_j^N,k), \\
& \nabla u^{i, N}_j(z_j^N)\cdot g^1(\omega, \delta, D^N_j, z_j^N) \nabla G (\Bx-z_j^N,k) \to 
\frac{1}{N} \Lambda \tilde g^1 (\nabla u(z_j^N)\cdot  d_j^N) (d_j^N \cdot \nabla G (\Bx-z_j^N,k)). 
\end{align*}

On the other hand, we have 
$$
u^N(x) \approx u^{in}+ \sum_{1\leq j \leq N} g^0(\omega, \delta, D^N_j)u^{i, N}_j(z_j^N) G (\Bx-z_j^N,k) +\sum_{1\leq j \leq N}\nabla u^{i, N}_j(z_j^N)\cdot g^1(\omega, \delta, D^N_j) \nabla G (\Bx-z_j^N,k). 
$$
By letting $N \to \infty$, we obtain
\begin{align} \label{eq-integral}
u(x) = u^{in} + \int_{\Omega} \Lambda \tilde g^0 u(y)\tilde V(y)G(x-y,k)dy
+ \int_{\Omega} \Lambda \tilde g^1 \nabla u (y)\tilde B \nabla G(x-y,k)dy. 
\end{align}

By applying the operator $\triangle + k^2$ to both side of the above equation, we get
$$
(\triangle + k^2)u (x) = \Lambda \tilde g^0 \tilde V u(x)  + \nabla \cdot \left(\Lambda  \tilde g^1 \tilde B \nabla u\right)(x), \quad \mbox{in } \Omega.
$$
Or equivalently,
\beq \label{eq-pde}
\nabla \cdot \left(I- \Lambda \tilde g^1 \tilde B \right)\nabla u(x) + (k^2- \Lambda\tilde g^0 \tilde V)u(x) =0, \quad \mbox{in } \Omega. 
\eeq
Therefore, we have shown that the microscopic field $u^N$ tends to a macroscopic field $u$ which satisfies the following effective equation
\beq \label{eq-pde2}
\nabla \cdot M_1(x) \nabla u(x) + M_2(x)u(x) =0, \quad \mbox{in } \mathbb{R}^3,
\eeq
where 
$$
M_1 = \begin{cases}
I, &\quad \mbox{in } \mathbb{R}^3\setminus \Omega,
 \\
I- \Lambda \tilde g^1 \tilde B, &\quad \mbox{in }   \Omega,
 \end{cases}
 $$
 and
 $$
 M_2 = \begin{cases}
 k^2, &\quad \mbox{in }  \mathbb{R}^3\setminus \Omega,
 \\
 k^2- \Lambda\tilde g^0 \tilde V, &\quad \mbox{in } \Omega.
 \end{cases}
$$

\begin{remark}
We have derived the effective media theory in a formal way under the crucial Assumption \ref{asum}. We leave a rigorous justification as an open problem for future investigation. We note that a justification is provided in 
 \cite{Ammari_Hai} for a related problem using the point interaction approximation. However the point scatterers therein only have monopole mode which is different from the case considered here where the point scatterers have monopole mode and dipole mode. We expect that new techniques are required 
for the justification. 
\end{remark}

\subsection{Double-negative refractive index media}

If the bubble dimers are distributed such that $\tilde B(x)$ is symmetric and positive definite with
$\tilde B(x) \geq C >0$ for some constant C for all $x \in \Omega$, then we see that for $\omega$ in the form
$\omega = \omega_{M,2} + as^2$ with $a < \mu^3 \hat \eta_1$, and sufficiently large $\Lambda$, both the matrix $Id- \Lambda \tilde g^1 \tilde B$ and the scalar function
$k^2- \Lambda\tilde g^0 \tilde V$ are negative. Therefore, we obtain an effective double-negative media with both negative mass density and negative bulk modulus. On the other hand, for $\omega$ in between $\omega_{M, 1}$ and $\omega_{M, 2}$ but away from the anti-resonance $\omega_{M, 2}$, $\tilde {g}^1$ may be small enough such that the matrix $Id- \Lambda \tilde g^1 \tilde B$ is positive, while the matrix $k^2- \Lambda\tilde g^0 \tilde V$ remains negative. Then the obtained effective media has negative mass density and positive bulk modulus.

\section{Numerical illustrations} \label{sec:num} 
In this section, we illustrate the double-negative refractive index phenomenon in bubbly media by numerical examples.

We consider a cubic array of identical spherical bubble dimers. Suppose $\Omega$ is a cube with a side length of $L=20$, {\it i.e.}, $\Omega=[0,20]^3$. 
Let $a=0.2$ and define a small cube $\Omega_a=[0,a]^3$. 
Then $\Omega$ can be considered as a union of small cubes as follows:
 $$
 \displaystyle\Omega=\bigcup_{n_1,n_2,n_3=0,1,...,99} \Omega_a + a(n_1,n_2,n_3).
 $$
We assume that a bubble dimer is centered in each of the small cubes. 
Then the total number of dimers is $N=10^6$ and the periodicity of the dimer array is $a=0.2$.


Recall that a bubble dimer is described by $z+s R_d D$, where $z$ is the center of the dimer, $s$ is its characteristic size, and $R_d$ is a rotation in $\RR^3$ which aligns the dimer in the direction $d$, where $d$ is a unit vector. We set the characteristic size of the dimers to be $s=0.1$. Since $D$ has unit volume, the radius $r_0$ of the bubbles comprising the dimers is $r_0=s(3/4\pi)^{1/3}\approx 0.005$. 

We set $\tilde{\rho}=\tilde{\kappa}=1$ and $\rho=\kappa=5\times 10^3$. Then $v=\tilde{v}=1$, $k=\tilde{k}=\omega$, and $\delta = 2\times 10^{-4}$. We assume that the two bubbles comprising each dimer are separated by a distance of $l = 5r_0\approx 0.0248$. Moreover, each dimer is randomly oriented so that the unit vector $d$ is uniformly distributed on the unit sphere. Under these assumptions, we can easily check that
$
\Lambda = 8\times 10^3$, $\tilde{V}\approx |\Omega|^{-1}=1.25\times 10^{-4} $, and $ \tilde{B}\approx (2|\Omega |)^{-1}I=6.25\times 10^{-5}I.
$

Now we consider the effective properties of the homogenized media. Recall that the effective coefficients of the homogenized equation \eqnref{eq-pde} are $k^2-\Lambda \tilde{g}^0 \tilde{V}$ and $I-\Lambda \tilde{g}^1 \tilde{B}$.
Using the above parameters, we have
\begin{align}
&k^2-\Lambda \tilde{g}^0 \tilde{V} \approx \omega^2-\tilde{g}^0 =\omega^2(1- \tilde{g}^0/\omega^2),\nonumber
\\
& I-\Lambda \tilde{g}^1 \tilde{B} \approx (1-\tilde{g}^1/2) I. 
\label{approx_eff_coeff}
\end{align}
Note that the coefficient $I-\Lambda \tilde{g}^1 \tilde{B}$ can be roughly considered as a scalar quantity. The scattering functions $\tilde{g}_0$ and $\tilde{g}_1$ can be computed as follows. Since $\tilde{g}_0\approx s g_0$ and $\tilde{g}_1\approx s g_1$, we have from
 \eqnref{sc_func_monopole} and \eqnref{sc_func_dipole} that
\begin{align*}
\tilde{g}^0(\omega, \delta, sR_dD) &\approx \frac{2(C_{11}+C_{12})}{1- \omega_1^2/\omega^2}, \\
\tilde{g}^1(\omega, \delta, sR_dD)&\approx  \frac{\delta \tilde v^2}{2(\omega_2^2-\omega^2)} P^2d_id_j.
\end{align*}

The resonance frequencies $\omega_1$ and $\omega_2$ can be easily calculated using a standard multipole expansion together with a root finding method such as Muller's method \cite{akl}. We find that $\omega_1 \approx 4.6171-0.0926i$ and $\omega_2 \approx 5.3253-0.0005 i$. Then it is simple to compute $\tilde{g}^0$ and $\tilde{g}^1$.

In the left of Figure \ref{fig:eff}, we plot the two effective coefficients as functions of frequency. Clearly, there is a narrow region contained in the interval $[5.2,5.4]$ in which both of the coefficients are negative. So we expect that the negative refraction occurs in this frequency region.

We next consider the refractive index.
In view of \eqnref{eq-pde} and \eqnref{approx_eff_coeff}, the effective mass density $\rho_{eff}$ and the effective bulk modulus $\kappa_{eff}$ can be computed approximately by
$$
\rho_{eff} \approx 1- \tilde{g}^1/2, \quad \kappa_{eff} \approx (1- \tilde{g}^0/\omega^2)^{-1}.
$$
As usual, we define the refractive index $n_{eff}$ by
$$
n_{eff}= {\sqrt \rho_{eff}} {\sqrt{\kappa_{eff}^{-1}}}.
$$
In the right figure of Figure \ref{fig:eff}, we plot the refractive index as a function of frequency. It is clear that the refractive index becomes negative in a narrow region contained in the interval $[5.2,5.4]$, as expected.

\begin{figure*}
\begin{center}
\epsfig{figure=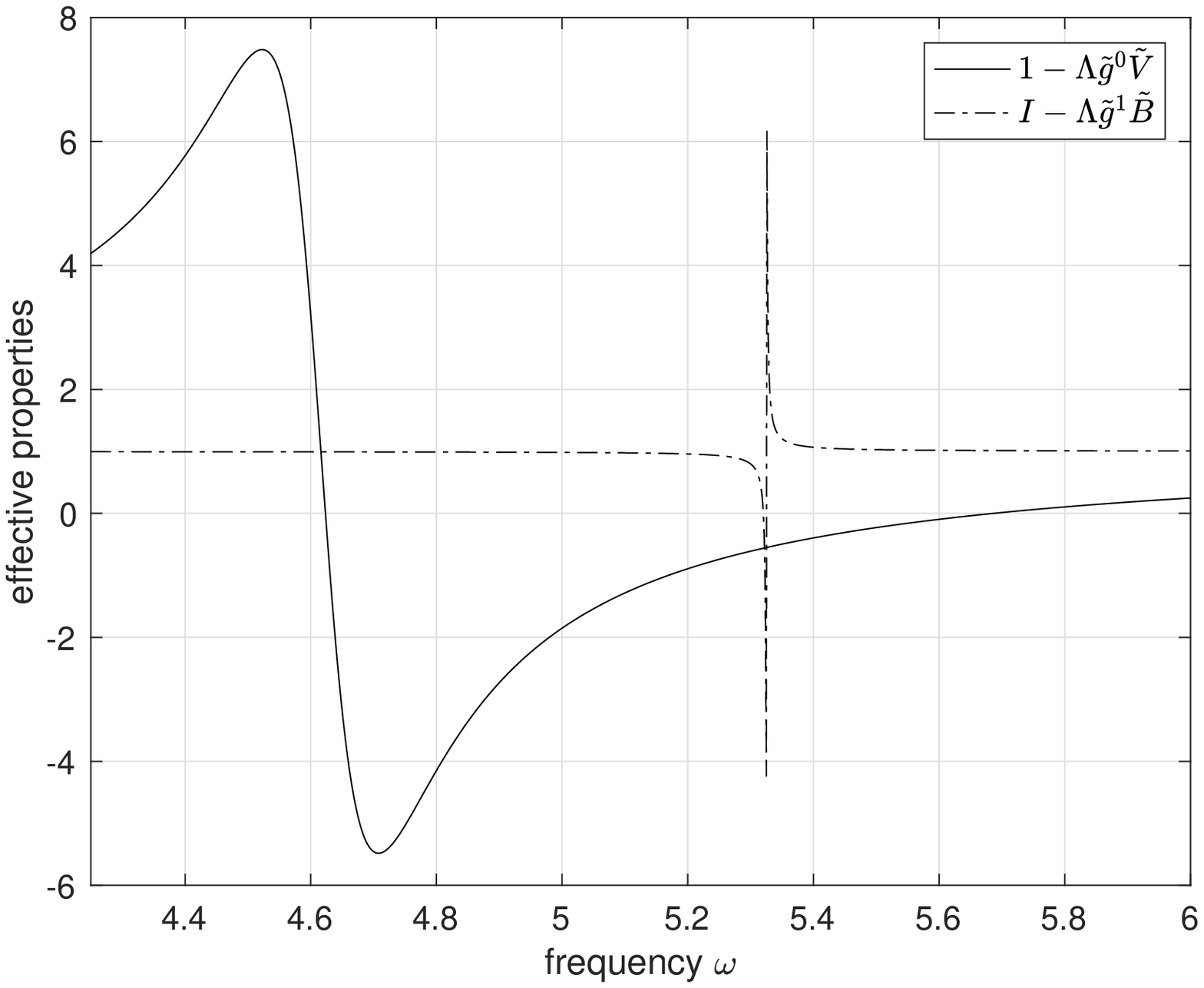,width=7cm}
\epsfig{figure=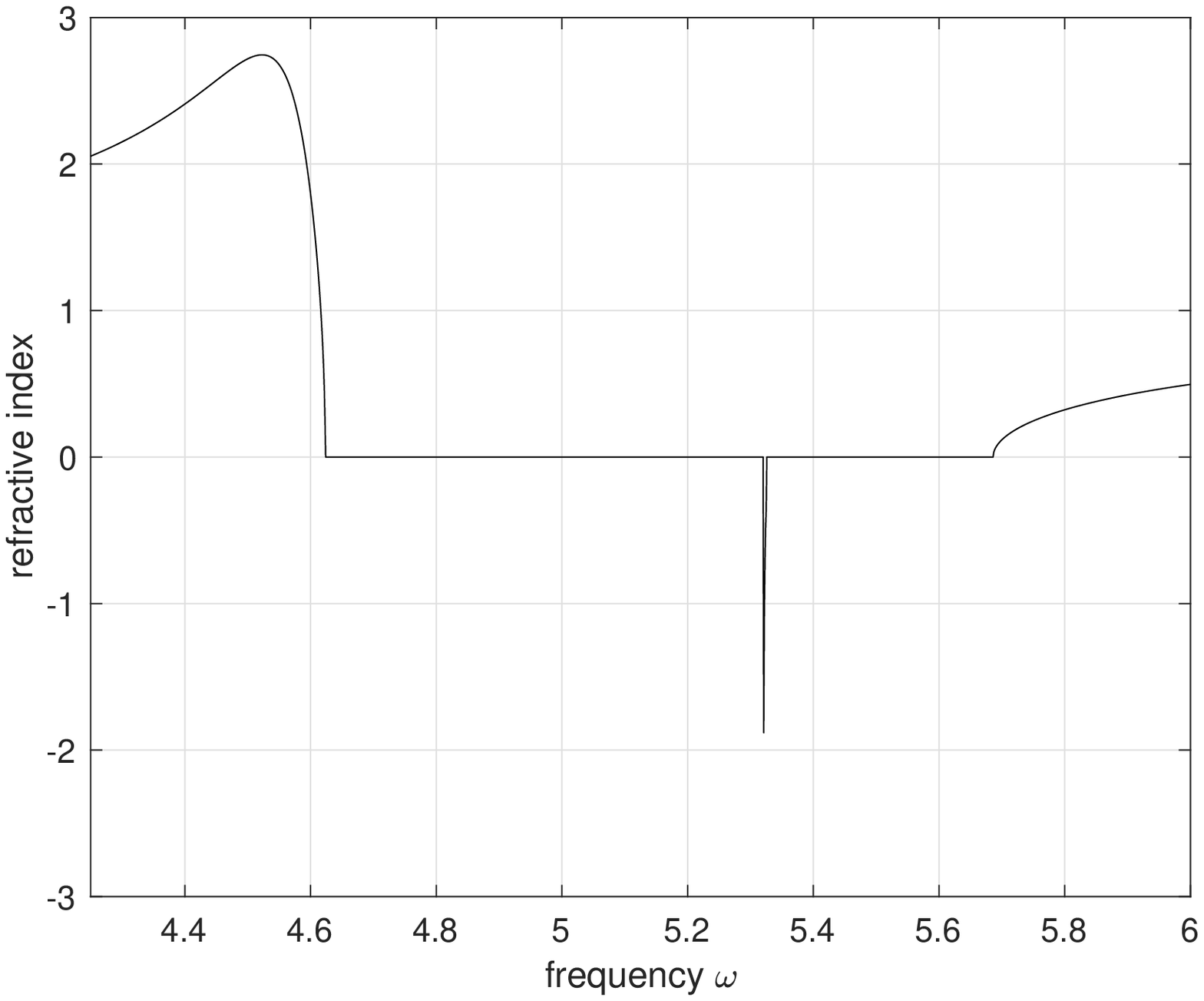,width=7cm}
\end{center}
\caption{The effective properties of the homogenized media (left), and the refractive index (right).}
\label{fig:eff}
\end{figure*}

%


\begin{thebibliography}{99}


\bibitem{pierre}
 H. Ammari, Y. Deng, and P. Millien.
Surface plasmon resonance of nanoparticles and applications in imaging, Arch. Ration. Mech. Anal., 220 (2016), 109--153. 

\bibitem{H3a}
H.~Ammari,  B. Fitzpatrick, D.~Gontier, H.~Lee, and H.~Zhang.
\newblock {M}innaert resonances for acoustic waves in bubbly media.
\newblock {arXiv:1603.03982}, 2016.

\bibitem{hamdache} H. Ammari, K. Hamdache, and J.C. N\'ed\'elec. 
Chirality in the Maxwell equations by the dipole approximation. 
SIAM J. Appl. Math.,  59 (1999), 2045--2059.



 \bibitem{AFLYZ} H. Ammari, B. Fitzpatrick, H. Lee, S. Yu, and H. Zhang. Subwavelength phononic bandgap opening in bubbly media. J. Diff. Equat., 263 (2017), 5610--5629.
 
 
\bibitem{Ammari_David} H. Ammari, B. Fitzpatrick, D. Gontier and H. Lee and H. Zhang.
A mathematical and numerical framework for bubble meta-screens.
SIAM J. Appl. Math.,  77 (2017),  1827--1850.
 
 \bibitem{Ammari_David2} H. Ammari, B. Fitzpatrick, D. Gontier and H. Lee and H. Zhang.
 {Sub-wavelength focusing of acoustic waves in bubbly media}.
{Proc. Royal Soc. A}, 473 (2017), 20170469. 
 
\bibitem{book2} {H. Ammari and H. Kang}. \textsl{Polarization and Moment
Tensors with Applications to Inverse Problems and Effective Medium
Theory}, Applied Mathematical Sciences, Vol. 162, Springer-Verlag,
New York, 2007.

\bibitem{AK04} {H. Ammari and H. Kang}. Boundary layer techniques for solving the Helmholtz equation
in the presence of small inhomogeneities. J. Math. Anal. Appl., 296 (2004),
190--208.

\bibitem{akl}
H.~Ammari, H.~Kang, and H.~Lee.
\newblock {\em Layer Potential Techniques in Spectral Analysis}, Mathematical Surveys and Monographs, volume 153,
\newblock American Mathematical Society Providence, 2009.

\bibitem{arma}
H.~Ammari, H.~Kang, and H.~Lee.
\newblock Asymptotic analysis of high-contrast phononic crystals and a
  criterion for the band-gap opening.
\newblock {Arch. Ration. Mech. Anal.}, 193 (2009), 679--714.
 


\bibitem{homogenization} H.~Ammari, H.~Lee, and H. Zhang.  
    High frequency homogenization of bubbly crystals,  arXiv:1708.07955. 

\bibitem{matias1}
H.~Ammari, P.~Millien, M.~Ruiz, and H.~Zhang.
\newblock Mathematical analysis of plasmonic nanoparticles: the scalar case.
\newblock {Arch. Ration. Mech. Anal.}, 224 (2017), 597--658.

\bibitem{matias2}
H.~Ammari, M.~Ruiz, S.~Yu, and H.~Zhang.
\newblock Mathematical analysis of plasmonic resonances for nanoparticles: the
  full Maxwell equations.
\newblock {J. Differ. Equat.}, 261 (2016), 3615--3669.


 \bibitem{Ammari_Hai}
H.~Ammari and H.~Zhang.
\newblock Effective medium theory for acoustic waves in bubbly fluids near
  Minnaert resonant frequency.
\newblock SIAM J. Math. Anal., 49(2017), 3252--3276. 

\bibitem{Ammari2015_a}
H.~Ammari and H.~Zhang.
\newblock Super-resolution in high-contrast media.
\newblock { Proc. R. Soc. A}, 471 (2015), 2178.

%




\bibitem{alu} S.A. Cummer, J. Christensen, and A. Al\`u. Controlling sound with acoustic metamaterials. Nature Rev., 1 (2016), 16001. 

\bibitem{figari} R.\,Figari, G.\,Papanicolaou and J.\,Rubinstein. \textit{Remarks on the point interaction approximation}, Hydrodynamic Behavior and Interacting Particle Systems, G. Papanicolaou (ed.), Springer-Verlag New York Inc. 1987.

\bibitem{figotin}
A.~Figotin and P.~Kuchment.
\newblock Spectral properties of classical waves in high-contrast periodic
  media.
\newblock {SIAM J. Appl. Math.}, 58 (1998), 683--702.



\bibitem{rev1} N. Kaina, F. Lemoult, M. Fink, and G. Lerosey. Negative refractive index and acoustic superlens from multiple scattering in single negative metamaterials, Nature, 525 (2015), 77--81L. 


\bibitem{Lek} J. Lekner. Capacitance coefficients of two spheres, Journal of Electrostatics,
69 (2011), 11--14



\bibitem{fink}  F. Lemoult, N. Kaina, M. Fink, and G. Lerosey. Soda cans metamaterial: a subwavelength-scaled photonic crystal. Crystals, 6 (2016), 82. 


\bibitem{leroy2} M.\,Lanoy, R.\,Pierrat, F.\,Lemoult, M.\,Fink, V.\,Leroy, and A.\,Tourin. Subwavelength focusing in bubbly media using broadband time reversal. Phys. Rev. B, 91.22 (2015), 224202.

\bibitem{leroy1}
V.~Leroy, A.~Bretagne, M.~Fink, H.~Willaime, P.~Tabeling, and A.~Tourin.
\newblock Design and characterization of bubble phononic crystals.
\newblock {Appl. Phys. Lett.}, 95 (2009), 171904.

\bibitem{leroy3} V.\,Leroy, A.\,Strybulevych, M.\,Lanoy, F.\,Lemoult, A.\,Tourin, and J. H.\,Page. Superabsorption of acoustic waves with bubble metascreens. Phys. Rev. B, 91.2 (2015), 020301.


\bibitem{rev2} G. Ma and P. Sheng. Acoustic metamaterials: From local resonances to broad horizons, Sci. Adv., 2 (2016), e1501595. 

\bibitem{milton2007} G.\,W.\,Milton, N.\,A.\,P.\,Nicorovici, and R.\,C.\,McPhedran. Opaque perfect lenses. Phys. B, 394 (2007), 171--175.


\bibitem{Minnaert1933}
M.~Minnaert.
\newblock  {O}n musical air-bubbles and the sounds of running water.
\newblock {The London, Edinburgh, Dublin Philos. Mag. and J. of Sci.},
  16 (1933), 235--248.
  
  

\bibitem{milton1994} N.\,A.\,Nicorovici, R.\,C.\,McPhedran, and G.\,W.\,Milton. Optical and dielectric properties of partially resonant composites. Phys. Rev. B, 49.12 (1994), 8479--8482.

\bibitem{ozawa} S.\,Ozawa. \textit{Point interaction potential approximation for $(-\triangle + U)^{-1}$ and eigenvalues of the Laplacian on wildly perturbed domain}. Osaka J. Math. 20 (1983), 923--937. 


	\bibitem{papanicolaou} G.C.\,Papanicolaou. \textit{Diffusion in random media}, Surveys in Applied Mathematics, volume 1, Edited by J.P.\,Keller, D. W.\,McLaughlin and G.C.\,Papanicolaou, Plenum Press, New York, 1995.  
	
\bibitem{pendry2000} J.\,B.\,Pendry. Negative refraction makes a perfect lens. Phys. Rev. Lett., 85 (2000), 3966--3969.


\bibitem{revR} S.A. Ramakrishna. Physics of negative refractive index materials, 
Rep. Progr. Phys., 68 (2005), 449. 



\bibitem{rev3} R. Zhu, X.N. Liu, G.K. Hu, C.T. Sun, and G.L. Huang. Negative refraction of elastic waves at the deep-subwavelength scale in a single-phase metamaterial. Nature Comm., 5(2014), 5510.










%
%
%
%


\end{thebibliography}
\end{document}